\documentclass[preprint, 10pt, english]{elsarticle}
\usepackage{amsthm}
\usepackage{amsmath}
\usepackage{latexsym, amssymb}
\usepackage{txfonts}
\usepackage{mathtools}
\usepackage{color}
\usepackage{babel}
\usepackage[all]{xy}
\usepackage[capitalise]{cleveref}
\usepackage{todonotes}

\newtheorem{thm}{Theorem}[section] 

\newtheorem{cor}[thm]{Corollary}

\newtheorem{lem}[thm]{Lemma}

\theoremstyle{definition}
\newtheorem{rem}[thm]{Remark}

\newcommand\operA[2]{{\if!#2!\operatorname{#1}\else{\operatorname{#1}_{#2}^{\phantom{I}}}\fi}} 

\newcommand{\Trace}[1][]{\if!#1!\operatorname{Tr}\else{\operatorname{Tr}_{#1}^{\phantom{I}}}\fi} 

\long\def\forget#1\forgotten{{}} %

\def\({\left(}
\def\){\right)}

\newcommand\LAY[3][]{{\begin{array}{c}\mbox{#2} \if#1!{}\else{+}\fi \\ \mbox{#3}\end{array}}}

\makeatletter

\def\ps@pprintTitle{%
 \let\@oddhead\@empty
 \let\@evenhead\@empty
 \def\@oddfoot{}%
 \let\@evenfoot\@oddfoot}

\newcommand{\bigperp}{%
  \mathop{\mathpalette\bigp@rp\relax}%
  \displaylimits
}

\newcommand{\bigp@rp}[2]{%
  \vcenter{
    \m@thHbox{\scalebox{\ifx#1\displaystyle2.1\else1.5\fi}{$#1\perp$}}
  }%
}
\makeatother

\renewcommand{\geq}{\geqslant}
\renewcommand{\leq}{\leqslant}

\newif\iffurther
\furtherfalse

\begin{document}
\begin{frontmatter}

\title{Classes in $H_{p^m}^{n+1}(F)$ of lower exponent}

\author{Adam Chapman}
\ead{adam1chapman@yahoo.com}
\address{School of Computer Science, Academic College of Tel-Aviv-Yaffo, Rabenu Yeruham St., P.O.B 8401 Yaffo, 6818211, Israel}

\author{Daniel Krashen}
\ead{dkrashen@upenn.edu}
\address{Department of Mathematics, University of Pennsylvania, Philadelphia, Pennsylvania}

\author{Kelly McKinnie}
\address{Department of Mathematical Sciences, University of Montana, Missoula, MT 59812, USA}
\ead{kelly.mckinnie@umontana.edu}

\begin{abstract}
Let $F$ be a field of characteristic $p>0$. We prove that if a symbol $A=\omega \otimes \beta_1 \otimes \dots \otimes \beta_n$ in $H_{p^m}^{n+1}(F)$ is of exponent dividing $p^{m-1}$, then its symbol length in $H_{p^{m-1}}^{n+1}(F)$ is at most $p^n$. In the case $n=2$ we also prove that if $A= \omega_1\otimes \beta_1+\cdots+\omega_r\otimes \beta_r$ in $H_{p^{m}}^2(F)$
satisfies $\exp(A)|p^{m-1}$, then the symbol length of $A$ in $H_{p^{m-1}}^2(F)$ is at most $p^r+r-1$. We conclude by looking at the case $p=2$ and proving that if $A$ is a sum of two symbols in $H_{2^m}^{n+1}(F)$ and $\exp A |2^{m-1}$, then the symbol length of $A$ in $H_{2^{m-1}}^{n+1}(F)$ is at most $(2n+1)2^n$. Our results use norm conditions in characteristic $p$ in the same manner as Matrzi in his paper ``On the symbol length of symbols''.

\end{abstract}

\begin{keyword}
Cyclic Algebras, Brauer Group, Kato-Milne Cohomology, Symbol Length, Fields of Positive Characteristic
\MSC[2020] 19D45 (primary); 11E04, 11E81, 16K20 (secondary)
\end{keyword}
\end{frontmatter}

\section{Introduction}

In \cite{Matzri:2024} it was proven that when $p$ is a prime integer and $F$ is a $p$-special field of $\operatorname{char}(F)=0$, every symbol $A=(a_1,\dots,a_n,a_{n+1})_{p^m,F} \in H^{n+1}(F,\mu_{p^m}^{\otimes n})$ of $\exp(A) | p^{m-1}$ is of symbol length at most $p^n$ in $H^{n+1}(F,\mu_{p^{m-1}}^{\otimes n})$. We wish to prove the analogous result for $H^{n+1}_{p^m}(F)$ when $\operatorname{char}(F)=p$. We start with a simple lemma that allows us to deduce the required norm condition. 

\begin{lem}\label{ourlemma}
Suppose $\beta_1,\dots,\beta_n \in F^\times$ are not $p$-independent. Then, the form 
\[\sum_{\vec{d}\in V \setminus \{\vec{0}\}} \beta_1^{d_1} \dots \beta_n^{d_n} x_{\vec{d}}^p\] 
is universal in $F^p$, i.e., every element $\gamma^p$ in $F^p$ is represented by this form.
\end{lem}

\begin{proof}
The $p$-dependence means $\sum_{\vec{d} \in V} \beta_1^{d_1} \dots \beta_n^{d_n} x_{\vec{d}}^p=0$ for some $x_{\vec{d}} \in F$, not all zero.
Take $\vec{d} \in V$ for which $x_{\vec{d}} \neq 0$, and multiply the equation by $(\beta_1^{d_1}\dots\beta_n^{d_n})^{-1}$. This gives a new equation of the same type with $x_{\vec{0}}\neq 0$.
Then
$\sum_{\vec{d}\in V \setminus \{\vec{0}\}} \beta_1^{d_1} \dots \beta_n^{d_n} x_{\vec{d}}^p=-x_{\vec{0}}^p$, which means 
\[\sum_{\vec{d}\in V \setminus \{\vec{0}\}} \beta_1^{d_1} \dots \beta_n^{d_n} (-\gamma{x_{\vec{d}}}/{x_{\vec{0}}})^p=\gamma^p.\]
 \end{proof}

In our case, there is no need to assume $F$ is $p$-special because of the following norm condition for splitness of symbols in Kato-Milne cohomology in characteristic $p$:

\begin{lem}[Essentially {\cite{ArasonBaeza:2010}}]\label{ABlem}
The symbol $\alpha \otimes \beta_1 \otimes \dots \otimes \beta_n$ is trivial in $H_p^{n+1}(F)$ if and only if 
$$\alpha=\lambda^p-\lambda+\sum_{\begin{matrix}0\leq d_1,\dots,d_n \leq p-1\\ (d_1,\dots,d_n)\neq (0,\dots,0)\end{matrix}} z_{d_1,\dots,d_n}^p \beta_1^{d_1} \dots \beta_n^{d_n}$$ for some $\lambda$ and $z_{d_1,\dots,d_n} \in F$.
In particular, $[\alpha,\beta)_{p,F}$ is split when $\alpha=\lambda^p-\lambda+ z_1^p \beta+\dots+z_{p-1}^p \beta^{p-1}$ for some $\lambda,z_1,\dots,z_{p-1}\in F$.
\end{lem}
\begin{proof} The remark in the cited paper covers the case when $\beta_1,\ldots,\beta_n$ are $p$-independent. When the $\beta_i$'s are not $p$-independent, Lemma \ref{ourlemma} can be applied so that the $z$'s are chosen so that 
\[\sum_{\begin{matrix}0\leq d_1,\dots,d_n \leq p-1\\ (d_1,\dots,d_n)\neq (0,\dots,0)\end{matrix}} z_{d_1,\dots,d_n}^p \beta_1^{d_1} \dots \beta_n^{d_n}=\alpha^p.\]
Taking $\lambda = -\alpha$, we get our desired result.

\end{proof}

This norm condition for the splitness of symbols in $H_p^{n+1}(F)$ allows us to prove a symbol length bound in $H^{n+1}_{p^{m-1}}(F)$ (Theorem \ref{T3}) in a similar way to what was done in \cite{ChapmanMcKinnie:2023} and \cite{ChapmanFlorenceMcKinnie:2023}. The same tool allows us to prove that if $A=\omega_1\otimes \beta_1+\cdots+\omega_r\otimes \beta_r$ in $H_{p^{m}}^2(F)$ satisfies $\exp(A)|p^{m-1}$, then the symbol length of $A$ in $H_{p^{m-1}}^2(F)$ is at most $p^r+r-1$ (Theorem \ref{T4}).
We finish by providing the analogous result to \cite[Section 9]{Matzri:2024}; If $A$ is a sum of two symbols in $H_{2^m}^{n+1}(F)$ and $\exp A |2^{m-1}$, then the symbol length of $A$ in $H_{2^{m-1}}^{n+1}(F)$ is at most $(2n+1)2^n$ (Theorem \ref{T5}).
\section{Preliminaries}

Let $p$ be a prime integer and $F$ a field of $\operatorname{char}(F)= p$.
We denote by $W_mF$ the ring of truncated Witt vectors over $F$ of length $m$.
Following \cite{Izhboldin:2000} and \cite{AravireJacobORyan:2018}, the cohomology groups $H_{p^m}^{n+1}(F)$ are defined to be $W_m F\otimes \underbrace{F^\times \otimes \dots \otimes F^\times}_{n \ \text{times}}$ modulo the relations 
\begin{itemize}
\item $(0,\dots,0,\beta_i,0,\dots,0) \otimes \beta_1 \otimes \dots \otimes \beta_n=0$, 
\item $\omega \otimes \beta_1 \otimes \dots \otimes \beta_n$ is trivial when $\beta_i=\beta_j$ for some $i \neq j$, and 
\item $(\omega^p-\omega) \otimes \beta_1 \otimes \dots \otimes \beta_n=0$ where $\omega^p$ is the Witt vector $(\omega_1^p,\dots,\omega_m^p)$ obtained by raising to the $p$th power each slot of the Witt vector $\omega=(\omega_1,\dots,\omega_m)$.
\end{itemize}

Note that because of the tensor product definition of the symbols
\begin{itemize}
\item $\omega \otimes \beta_1\otimes \cdots \otimes \beta_n + \omega' \otimes \beta_1\otimes \cdots \otimes \beta_n =
(\omega + \omega') \otimes \beta_1\otimes \cdots \otimes \beta_n$
\item $\omega \otimes \beta_1\otimes\cdots \otimes \beta_i\beta_i'\otimes \cdots \otimes \beta_n = 
\omega \otimes \beta_1\otimes\cdots \otimes \beta_i \otimes \cdots \otimes \beta_n + 
\omega \otimes \beta_1\otimes\cdots \otimes \beta_i' \otimes \cdots \otimes \beta_n$
\end{itemize}
Each $\omega \otimes \beta_1 \otimes \dots \otimes \beta_n$ is a ``symbol", and the symbol length of a class in $H^{n+1}_{p^m}(F)$ is the minimal number of symbols required to express it.

The following sequence is exact \cite[Theorem 2.31 together with Remark 2.32]{AravireJacobORyan:2018}: 
\begin{equation}0 \rightarrow H_{p^m}^{n+1}(F) \rightarrow H_{p^{m+\ell}}^{n+1}(F) \rightarrow H_{p^\ell}^{n+1}(F) \rightarrow 0
\label{e1}
\end{equation}
where the injective map sends each $(\omega_1,\dots,\omega_m) \otimes \beta_1 \otimes \dots \otimes \beta_n$ to $(\underbrace{0,\dots,0}_{\ell \ \text{times}},\omega_1,\dots,\omega_m) \otimes \beta_1 \otimes \dots \otimes \beta_n$ and the surjective map sends each $(\omega_1,\dots,\omega_{m+\ell})\otimes \beta_1 \otimes \dots \otimes \beta_n$ to $(\omega_1,\dots,\omega_\ell) \otimes \beta_1 \otimes \dots \otimes \beta_n$.
The group $H_{p^m}^{n+1}(F)$ can therefore be considered a subgroup of $H_{p^{m+\ell}}^{n+1}(F)$, and when $H_{p^\ell}^{n+1}(F)$ is considered as such in the sequence above, the surjective map coincides with the exponentiation map taking each symbol $S$ to $\underbrace{S+\dots+S}_{p^m \ \text{times}}$. One can therefore consider the symbol length in $H_{p^\ell}^{n+1}(F)$ of a class in $H_{p^{m+\ell}}^{n+1}(F)$ of exponent $p^\ell$.

These groups coincide with familiar groups for special $m$ and $n$. In particular, $H_{p^m}^2(F) \cong {_{p^m}Br}(F)$ with $\omega \otimes \beta \mapsto [\omega,\beta)_{p^m,F}$, where the latter stands for the cyclic algebra generated over $F$ by $\theta_1,\dots,\theta_m,y$ subject to the relations 
\[\theta^p-\theta=\omega, \, y^{p^m}=\beta \textrm{, and } (y \theta_1 y^{-1},y \theta_2 y^{-1}\dots,y \theta_m y^{-1})=\theta+(1,0,\dots,0)\]
where  $\theta=(\theta_1,\dots,\theta_m)$ and $\theta^p=(\theta_1^p,\dots,\theta_m^p)$ where the arithmetic operations (multiplication, addition and subtraction) obey the rules of Witt vectors.

\section{Single symbols in $H_{p^m}^{n+1}(F)$}

We begin with a lemma outlining the rules with which we manipulate the symbols in $H_{p^m}^{n+1}(F)$ in the proof of Theorem \ref{T3}.

\begin{lem}\label{ourlemma2} Given $\vec{d}\in \{0,1,\dots,p^m-1\}^{\times n}$ where $d_i$ is prime to $p$ at least for one $i$ in $\{1,\dots,n\}$, the symbol $\omega \otimes \beta_1 \otimes \beta_2 \otimes \dots \otimes \beta_n$ in $H_{p^m}^{n+1}(F)$ can be written as $\omega \otimes (\prod_{k=1}^n \beta_k^{d_k}) \otimes \gamma_2 \otimes \dots \otimes \gamma_n$ for some $\gamma_2,\dotsm,\gamma_n$.
\end{lem}

\begin{proof} We show it here for $n=2$ and the general argument readily follows. Using the relations laid out in the preliminary section, we can assume without loss of generality that $\gcd(d_1,p)=1$. This is  because $\omega \otimes \beta_1 \otimes \beta_2=\omega \otimes \beta_2 \otimes \beta_1^{-1}$ which enables moving $\beta_i$ with $d_i$ prime to $p$ to the appropriate location. Similarly, the rule $\omega\otimes \beta_1^{d_1} \otimes \beta_2^{c} = \omega \otimes \beta_1 \otimes \beta_2$ for $c$ with $d_1 c\equiv 1 \pmod{p^m}$ allows us to reduce to the case of $d_1=1$. Lastly, the rule $\omega \otimes \beta_1 \otimes \beta_2=\omega \otimes \beta_1 \beta_2^{d_2} \otimes \beta_2$ completes the picture.
\end{proof}

\begin{thm}\label{T3}
Given a field $F$ of $\operatorname{char}(F)=p>0$, the symbol length in $H_{p^{m-1}}^{n+1}(F)$ of a single symbol in $H_{p^m}^{n+1}(F)$ of exponent $p^{m-1}$ is at most $p^n$.
\end{thm}

\begin{proof}
The case of $n=1$ coincides with \cite[Proposition 5]{MammoneMerkurjev:1991}. We continue with $n\geq 2$, although the proof could be adapted to suit the case of $n=1$ as well.
Consider the symbol $\omega \otimes \beta_1 \otimes \dots \otimes \beta_n$ in $H_{p^m}^{n+1}(F)$ and suppose that it is of exponent $p^{m-1}$.
Then $\omega_1 \otimes \beta_1 \otimes \dots \otimes \beta_n$ is trivial in $H_p^{n+1}(F)$.
By Lemma \ref{ABlem}, $$\omega_1=\lambda^p-\lambda+\sum_{\begin{matrix}0\leq d_1,\dots,d_n \leq p-1\\ (d_1,\dots,d_n)\neq (0,\dots,0)\end{matrix}} z_{d_1,\dots,d_n}^p \beta_1^{d_1} \dots \beta_n^{d_n}$$ for some $\lambda$ and $z_{d_1,\dots,d_n} \in F$.
Then $\omega \otimes \beta_1 \otimes \dots \otimes \beta_n$ can be written as 
\begin{eqnarray*}
\left(\omega-(\lambda^p,0,\dots,0)+(\lambda,0,\dots,0)-\sum (z_{d_1,\dots,d_n}^p \beta_1^{d_1}\dots \beta_n^{d_n},0,\dots,0)\right) \otimes \beta_1 \otimes \dots \otimes \beta_n+\\
\sum (z_{d_1,\dots,d_n}^p \beta_1^{d_1}\dots \beta_n^{d_n},0,\dots,0) \otimes \beta_1 \otimes \dots \otimes \beta_n.
\end{eqnarray*}
The first term has 0 in the first slot of the Witt vector, and thus is a single symbol in $H_{p^{m-1}}^{n+1}(F)$.
Now, each term $(z_{d_1,\dots,d_n}^p \beta_1^{d_1}\dots \beta_n^{d_n},0,\dots,0) \otimes \beta_1 \otimes \dots \otimes \beta_n$ can be written as
$(z_{d_1,\dots,d_n}^p \gamma_1,0,\dots,0) \otimes \gamma_1 \otimes \dots \otimes \gamma_n$ where $\gamma_1=\beta_1^{d_1}\dots \beta_n^{d_n}$ and some choice of $\gamma_2,\dots,\gamma_n$ by Lemma \ref{ourlemma2}. Thus, 

\begin{multline*}
(z_{d_1,\dots,d_n}^p \gamma_1,0,\dots,0) \otimes \gamma_1 \otimes \gamma_2 \otimes\dots \otimes \gamma_n=(z_{d_1,\dots,d_n}^p \gamma_1,0,\dots,0) \otimes (z_{d_1,\dots,d_n}^p \gamma_1) \otimes \gamma_2 \otimes \dots \otimes \gamma_n\\
-(z_{d_1,\dots,d_n}^p \gamma_1,0,\dots,0) \otimes z_{d_1,\dots,d_n}^p \otimes \gamma_2 \otimes \dots \otimes \gamma_n. 
\end{multline*}
The first term is 0 in the cohomology group. The second term, $-(z_{d_1,\dots,d_n}^p \gamma_1,0,\dots,0) \otimes z_{d_1,\dots,d_n}^p \otimes \gamma_2 \otimes \dots \otimes \gamma_n$, is the sum of $p$ copies of the single symbol $-(z_{d_1,\dots,d_n}^p \gamma_1,0,\dots,0) \otimes z_{d_1,\dots,d_n} \otimes \gamma_2 \otimes \dots \otimes \gamma_n$ in $H_{p^{m}}^{n+1}(F)$. Adding the Witt vectors instead of multiplying in the second slot gives $-(0, z_{d_1,\dots,d_n}^p \gamma_1,0,\ldots,0)\otimes z_{d_1,\dots,d_n} \otimes \gamma_2 \otimes \dots \otimes \gamma_n$. By Equation \ref{e1}, this symbol is $-(z_{d_1,\dots,d_n}^p \gamma_1,0,\ldots,0)\otimes z_{d_1,\dots,d_n} \otimes \gamma_2 \otimes \dots \otimes \gamma_n\in H^{n+1}_{p^{m-1}}$. This completes the proof.
\end{proof}

This bound is much better than the analogous bounds (with $n=2$ or $p=2$) in \cite{ChapmanFlorenceMcKinnie:2023} that increased with $m$.

\section{Sums of symbols in $H_{p^m}^2(F)$}

\begin{thm}\label{T4}
Suppose $A=\omega_1\otimes \beta_1+\cdots+\omega_r\otimes \beta_r$ in $H_{p^{m}}^2(F)$  satisfies $\exp(A)|p^{m-1}$, then the symbol length of $A$ in $H_{p^{m-1}}^2(F)$ is at most $p^r+r-1$.
\end{thm}

\begin{proof}
By induction on $r$. For $r=1$ this coincides with \cite[Proposition 5]{MammoneMerkurjev:1991}. Write $L=F[t_1,\dots,t_{r-1} : t_1^p=\beta_1,\dots,t_{r-1}^p=\beta_{r-1}]$, a purely inseparable field extension of exponent $p$.
Write $\alpha_1,\dots,\alpha_r$ for the initial slots of the Witt vectors $\omega_1,\dots,\omega_r$. Let $B = A^{p^{m-1}} = \alpha_1\otimes \beta_1+\cdots+\alpha_r\otimes \beta_r$. Since $\exp(A)| p^{m-1}$,  $B$ is trivial in $H_p^2(F)$, which means $\alpha_r \otimes \beta_r$ is trivial in $H_p^2(L)$.
Therefore, by Lemma \ref{ABlem},
$$\alpha_r=\lambda^p-\lambda+\sum_{k=1}^{p-1} x_k^p \beta_r^k$$
for some $\lambda,x_1,\dots,x_{p-1} \in L$. Note that all elements on both sides of the equality are clearly in $F$ except $\lambda$, which means $\lambda \in F$ as well.
Now, each $x_k$ is equal to $\sum_{i_1=0}^{p-1} \dots \sum_{i_{r-1}=0}^{p-1} z_{i_1,\dots,i_{r-1}} t_1^{i_1} \dots t_{r-1}^{i_{r-1}}$ for some $z_{i_1,\dots,i_{r-1}} \in F$. Therefore 
$$\alpha_r=\lambda^p-\lambda+\sum_{k=1}^{p-1} \sum_{i_1=0}^{p-1} \dots \sum_{i_{r-1}=0}^{p-1} z_{i_1,\dots,i_{r-1}}^p \beta_1^{i_1} \dots \beta_{r-1}^{i_{r-1}} \beta_r^k.$$
Write $A$ as 
$$A - \left((\lambda^p,0,\dots,0)-(\lambda,0,\dots,0)+\sum_{k=1}^{p-1} \sum_{i_1=0}^{p-1} \dots \sum_{i_{r-1}=0}^{p-1} (z_{i_1,\dots,i_{r-1}}^p \beta_1^{i_1} \dots \beta_{r-1}^{i_{r-1}} \beta_r^k,0,\dots,0)\right)\otimes \beta_r$$
$$+\sum_{k=1}^{p-1} \sum_{i_1=0}^{p-1} \dots \sum_{i_{r-1}=0}^{p-1} (z_{i_1,\dots,i_{r-1}}^p \beta_1^{i_1} \dots \beta_{r-1}^{i_{r-1}} \beta_r^k,0,\dots,0)\otimes \beta_r.$$

The term 
$$A - \left((\lambda^p,0,\dots,0)-(\lambda,0,\dots,0)+\sum_{k=1}^{p-1} \sum_{i_1=0}^{p-1} \dots \sum_{i_{r-1}=0}^{p-1} (z_{i_1,\dots,i_{r-1}}^p \beta_1^{i_1} \dots \beta_{r-1}^{i_{r-1}} \beta_r^k,0,\dots,0)\right)\otimes \beta_r$$ has 0 in the first slot of its Witt vector, hence can be written as a single symbol in $H_{p^{m-1}}^2(F)$.
Now, each $(z_{i_1,\dots,i_{r-1}}^p \beta_1^{i_1} \dots \beta_{r-1}^{i_{r-1}} \beta_r^k,0,\dots,0)\otimes \beta_r$ can be written as 
$$(z_{i_1,\dots,i_{r-1}}^p \beta_1^{i_1} \dots \beta_{r-1}^{i_{r-1}} \beta_r^k,0,\dots,0)\otimes z_{i_1,\dots,i_{r-1}}^p \beta_1^{i_1} \dots \beta_{r-1}^{i_{r-1}}\beta_r$$

$$-(z_{i_1,\dots,i_{r-1}}^p \beta_1^{i_1} \dots \beta_{r-1}^{i_{r-1}} \beta_r^k,0,\dots,0)\otimes z_{i_1,\dots,i_{r-1}}^p$$

$$-(z_{i_1,\dots,i_{r-1}}^p \beta_1^{i_1} \dots \beta_{r-1}^{i_{r-1}} \beta_r^k,0,\dots,0)\otimes \beta_1^{i_1} \dots \beta_{r-1}^{i_{r-1}}.$$
The first term of the three is trivial, and the second can be written as a single symbol in $H_{p^{m-1}}^2(F)$.
Altogether, we get that $A$ can be written as a sum of $1+(p-1)p^{r-1}$ symbols in $H_{p^{m-1}}^2(F)$ plus some  $B=\tau_1\otimes\beta+\dots+\tau_{r-1}\otimes\beta_{r-1} \in H_{p^m}^2(F)$ with $\exp(B)|p^{m-1}$. Thus, by the induction hypothesis, the symbol length of $B$ is at most $p^{r-1}+r-2$, and thus the symbol length of $A$ is at most $1+(p-1)p^{r-1}+p^{r-1}+r-2=p^r+r-1$.
\end{proof}

\begin{rem}
It is natural to compare the obtained bound to previous upper bounds from the literature. There is a mistake in \cite[Lemma 5.4 (a)]{ChapmanFlorenceMcKinnie:2023}: the argument treats $(z^p \beta,0,\dots,0)\otimes \beta$ as if it is trivial in $H_{p^m}^2(F)$, but it is not, it is merely a single symbol in $H_{p^{m-1}}^2(F)$. Fixing that, the bound obtained in \cite[Lemma 5.4 (a)]{ChapmanFlorenceMcKinnie:2023} is exactly $p$, just like in part $(b)$ of the same lemma and the corresponding result from \cite{MammoneMerkurjev:1991}. Taking that into consideration, the upper bound in \cite[Corollary 5.5 (a)]{ChapmanFlorenceMcKinnie:2023} on the symbol length in $H_{2^{m-1}}^2(F)$ of sums of two symbols in $H_{2^m}^2(F)$ is 6, whereas here we take it down to $2^2+1=5$. 
The upper bound in \cite[Corollary 5.5 (b)]{ChapmanFlorenceMcKinnie:2023} on the symbol length in $H_{3^{m-1}}^2(F)$ of sums of two symbols in $H_{3^m}^2(F)$ is should be similarly corrected to 15, whereas here it is $3^2+1=10$.
The bounds in \cite{ChapmanMcKinnie:2023} should be compared too: in \cite[Theorem 4.2]{ChapmanMcKinnie:2023} the argument suggests that the bound of the symbol length of $A$, a sum of four symbols in $H_{2^m}^2(F)$ of $\exp(A)|2^{m-1}$, is at most 8 times the symbol length of sums of two symbols in $H_{2^m}^2(F)$ with exponent dividing $2^{m-1}$, which means (taking the new bound into consideration) $8\cdot 5=40$ (and not 32 as written in the paper). However, the new bound that we present here is $2^4+4-1=19$, which is considerably lower. 
Similarly, the bound in \cite[Theorem 4.3]{ChapmanMcKinnie:2023} on the symbol length of  $A$, a sum of three symbols in $H_{2^m}^2(F)$ with $\exp(A)|2^{m-1}$, should be corrected to 15, whereas the new bound we present here suggests $2^3+3-1=10$. 
\end{rem}

\section{Sums of two symbols in $H_{2^m}^{n+1}(F)$}

In this section we make use of the isomorphism $H_2^{n+1}(F) \cong I_q^n F/I_q^{n+1} F$ given by 
\[
\alpha\otimes \beta_1\otimes \cdots \otimes \beta_n \mapsto 
\langle \! \langle \beta_n,\dots, \beta_1,\alpha ] \! ] 
\]
proven in \cite{Kato:1982}. For background on quadratic forms in the characteristic 2 case see \cite{EKM}. In this section the symbol $=$ between quadratic forms should be read as an `isometry' between the forms.

\begin{lem}[{\cite[Lemma 3.1]{ChapmanLevin:2024}}]\label{chainstep}
Given integers $n > k \geq 1$, a quadratic $n$-fold Pfister form $\psi$ and two quadratic $k$-fold factors $\varphi_1$ and $\varphi_2$, there exists a bilinear Pfister form $\rho$ for which $\psi=\rho \otimes \varphi_1$ and $\psi=\rho \otimes \varphi_2$.
\end{lem}

\begin{cor}\label{Chain}
If $\psi=\langle \! \langle a_1,\dots,a_n ] \! ] = \langle \! \langle b_1,\dots,b_n ] \! ]$, then there exist $c_1,\dots,c_{n-1} \in F^\times$ such that
$\varphi=\langle \! \langle c_1,\dots,c_i,a_{i+1},\dots,a_n ] \! ]=\langle \! \langle c_1,\dots,c_i,b_{i+1},\dots,b_n ] \! ]$ for any $i \in \{1,\dots,n-1\}$.
\end{cor}

\begin{proof}
By induction. For each $i\in \{1,\dots,n-1\}$, if we already have $\langle \! \langle c_1,\dots,c_{i-1},a_i,\dots,a_n ] \! ]=\langle \! \langle c_1,\dots,c_{i-1},b_i,\dots,b_n ] \! ]$, then set 
\begin{eqnarray*}
\varphi_1&=&\langle \! \langle c_1,\dots,c_{i-1},a_{i+1},\dots,a_n ] \! ], \quad \text{and}\\ \varphi_2&=&\langle \! \langle c_1,\dots,c_{i-1},b_{i+1},\dots,b_n ] \! ].
\end{eqnarray*}
 Then by Lemma \ref{chainstep}, there exists $\rho=\langle \! \langle c_i \rangle \! \rangle$ for which $\psi=\rho \otimes \varphi_1=\rho \otimes \varphi_2$, which means 
$\langle \! \langle c_1,\dots,c_i,a_{i+1},\dots,a_n ] \! ]=\langle \! \langle c_1,\dots,c_i,b_{i+1},\dots,b_n ] \! ]$.
\end{proof}

\begin{thm}\label{T5}
Let $A=\omega\otimes\beta_1\otimes\dots\otimes\beta_n-\tau\otimes \delta_1\otimes \dots\otimes \delta_n \in H_{2^m}^{n+1}(F)$ be a class with $\exp A|2^{m-1}$. Then the symbol length of $A$ in $H_{2^{m-1}}^{n+1}(F)$ is at most $(2n+1)2^n$. 
\end{thm}

\begin{proof}
By Corollary \ref{Chain}, there exist $\gamma_1,\dots,\gamma_n \in F^\times$ for which $\omega_1 \otimes \gamma_1\otimes \cdots \otimes \gamma_i \otimes \beta_{i+1} \otimes \dots \otimes \beta_n=\omega_1 \otimes \gamma_1 \otimes \cdots \otimes \gamma_i \otimes \delta_{i+1} \otimes \dots \otimes \delta_n \in H^{n+1}_2(F)$ for each $i\in \{1,\dots,n-1\}$.

Write $A$ as 
\begin{eqnarray*}
\omega\otimes\beta_1\otimes\dots\otimes\beta_n &-& \omega\otimes \gamma_1 \otimes \beta_2 \otimes \dots\otimes \beta_n \\
+\omega\otimes \gamma_1 \otimes \beta_2 \otimes \dots\otimes \beta_n&-&\omega\otimes \gamma_1 \otimes \gamma_2 \otimes \beta_3 \otimes \dots\otimes \beta_n\\
&\vdots& \\
+\omega\otimes \gamma_1 \otimes \dots \otimes \gamma_{n-2}  \otimes \beta_{n-1} \otimes \beta_n&-&\omega\otimes \gamma_1 \otimes \dots \otimes \gamma_{n-1} \otimes \beta_n\\
+\omega\otimes \gamma_1 \otimes \dots \otimes \gamma_{n-1} \otimes \beta_n & - & \omega\otimes \gamma_1 \otimes \dots \otimes \gamma_{n-1} \otimes \delta_n\\
+\omega\otimes \gamma_1 \otimes \dots \otimes \gamma_{n-1} \otimes \delta_n& -& \omega\otimes \gamma_1 \otimes \dots \otimes \gamma_{n-2}  \otimes \delta_{n-1} \otimes \delta_n\\
& \vdots &\\
+\omega\otimes \gamma_1 \otimes \gamma_2 \otimes \delta_3 \otimes \dots\otimes \delta_n &-& \omega\otimes \gamma_1 \otimes \delta_2 \otimes \dots\otimes \delta_n\\
+\omega\otimes \gamma_1 \otimes \delta_2 \otimes \dots\otimes \delta_n &-&\omega\otimes \delta_1 \otimes \dots\otimes \delta_n\\
+\omega\otimes \delta_1 \otimes \dots\otimes \delta_n &-&\tau\otimes \delta_1 \otimes \dots\otimes \delta_n.
\end{eqnarray*}
Each of the lines is an expression that can be written as a single symbol in $H_{2^m}^{n+1}(F)$ (because they share all slots except for one) of exponent dividing $2^{m-1}$ because their $2^{m-1}$th powers are isometric Pfister forms. Thus, the symbol length of the expression in each line in $H_{2^{m-1}}^{n+1}(F)$ is at most $2^n$ by Theorem \ref{T3}. Since there are $2n+1$ lines, the symbol length of $A$ in $H_{2^{m-1}}^{n+1}(F)$ is at most $(2n+1)2^n$.
\end{proof}

\bibliographystyle{abbrv}
\bibliography{bibfile}

\def\cprime{$'$}
\begin{thebibliography}{10}

\bibitem{ArasonBaeza:2010}
J.~K. Arason and R.~Baeza.
\newblock La dimension cohomologique des corps de type {${\bf C}_{\rm r}$} en
  caract\'eristique {${\bf p}$}.
\newblock {\em C. R. Math. Acad. Sci. Paris}, 348(3-4):125--126, 2010.

\bibitem{AravireJacobORyan:2018}
R.~Aravire, B.~Jacob, and M.~O'Ryan.
\newblock The de {R}ham {W}itt complex, cohomological kernels and
  {$p^m$}-extensions in characteristic {$p$}.
\newblock {\em J. Pure Appl. Algebra}, 222(12):3891--3945, 2018.

\bibitem{ChapmanFlorenceMcKinnie:2023}
A.~Chapman, M.~Florence, and K.~McKinnie.
\newblock Common splitting fields of symbol algebras.
\newblock {\em Manuscripta Math.}, 171(3-4):649--662, 2023.

\bibitem{ChapmanLevin:2024}
A.~Chapman and I.~Levin.
\newblock Invariant for sets of pfister forms, 2024.

\bibitem{ChapmanMcKinnie:2023}
A.~Chapman and K.~McKinnie.
\newblock Biquaternion algebras, chain lemma and symbol length.
\newblock {\em Mediterr. J. Math.}, 20(5):Paper No. 255, 8, 2023.

\bibitem{EKM}
R.~Elman, N.~Karpenko, and A.~Merkurjev.
\newblock {\em The algebraic and geometric theory of quadratic forms},
  volume~56 of {\em American Mathematical Society Colloquium Publications}.
\newblock American Mathematical Society, Providence, RI, 2008.

\bibitem{Izhboldin:2000}
O.~Izhboldin.
\newblock {$p$}-primary part of the {M}ilnor {$K$}-groups and {G}alois
  cohomologies of fields of characteristic {$p$}.
\newblock In {\em Invitation to higher local fields ({M}\"{u}nster, 1999)},
  volume~3 of {\em Geom. Topol. Monogr.}, pages 19--41. Geom. Topol. Publ.,
  Coventry, 2000.
\newblock With an appendix by Masato Kurihara and Ivan Fesenko.

\bibitem{Kato:1982}
K.~Kato.
\newblock Symmetric bilinear forms, quadratic forms and {M}ilnor {$K$}-theory
  in characteristic two.
\newblock {\em Invent. Math.}, 66(3):493--510, 1982.

\bibitem{MammoneMerkurjev:1991}
P.~{Mammone} and A.~{Merkurjev}.
\newblock {On the corestriction of $p^ n$-symbol}.
\newblock {\em {Isr. J. Math.}}, 76(1-2):73--80, 1991.

\bibitem{Matzri:2024}
E.~Matzri.
\newblock On the symbol length of symbols.
\newblock In {\em Amitsur {C}entennial {S}ymposium}, volume 800 of {\em
  Contemp. Math.}, pages 219--231. Amer. Math. Soc., [Providence], RI, [2024]
  \copyright 2024.

\end{thebibliography}

\end{document}